\documentclass[11pt]{amsart}%
\usepackage{enumerate}
\usepackage{amsfonts}
\usepackage{amsmath}
\usepackage{amssymb}
\usepackage{graphicx}%
\usepackage[final]{optional}
\usepackage[centertags]{amsmath}%
\usepackage{amsfonts,amssymb,amsthm, stmaryrd , graphicx, mathrsfs, mathpazo}%
\usepackage[usenames]{color}
\usepackage[toc,page]{appendix}

\newtheorem{thm}{Theorem}[section]
\newcommand{\bt}{\begin{thm}}
\newcommand{\et}{\end{thm}}
\newtheorem{cor}[thm]{Corollary} 
\newcommand{\bc}{\begin{cor}}
\newcommand{\ec}{\end{cor}}
\newtheorem{lem}[thm]{Lemma}  
\newcommand{\bl}{\begin{lem}}
\newcommand{\el}{\end{lem}}
\newtheorem{prop}[thm]{Proposition}
\newcommand{\bp}{\begin{prop}}
\newcommand{\ep}{\end{prop}}
\newtheorem{defn}[thm]{Definition}
\newcommand{\bd}{\begin{defn}}  
\newcommand{\ed}{\end{defn}}
\newtheorem{rmrk}[thm]{Remark}
\newcommand{\br}{\begin{rmrk}}
\newcommand{\er}{\end{rmrk}}
\newcommand{\be}{\begin{equation}}
\newcommand{\ee}{\end{equation}}
\newcommand{\R}{\mathbb{R}}

\newcommand{\Ric}{\operatorname{Ric}}
\newcommand{\sphere}{\mathbb{S}}
\newcommand{\diam}{\operatorname{diam}}
\newcommand{\vol}{\operatorname{vol}}
\newcommand{\dvol}{~\mathrm{d}vol}
\newcommand{\dsigma}{{~\mathrm{d}}\sigma}
\newcommand{\dr}{{~\mathrm{d}}r}
\newcommand{\dx}{{~\mathrm{d}}x}
\newcommand{\dy}{{~\mathrm{d}}y}

\newcommand{\n}{\textbf{n}}
\newcommand{\tr}{\operatorname{tr}}
\newcommand{\supp}{\operatorname{supp}}

\newcommand{\Rm}{\operatorname{Rm}}

\newcommand{\Lip}{\operatorname{Lip}}
\newcommand{\del}{\partial}

\newcommand{\pr}{\textrm{pr}}

\newcommand{\Haus}{\mathcal{H}}
\newcommand{\Id}{\mathrm{Id}}

\newcommand{\D}{\mathcal{D}}
\newcommand{\A}{\mathcal{A}}
\newcommand{\AK}{\mathcal{AK}}

\newcommand{\MM}{{M_1 \sqcup M_2}}

\begin{document}
\title[Neckpinch and Optimal Transport]{On the Size of a Ricci Flow Neckpinch via Optimal Transport}
\author{Sajjad Lakzian}
\address[Sajjad Lakzian]{HCM, Universit\"{a}t Bonn}
\email{SLakzian@gc.cuny.edu}
\author{Michael Munn}
\address[Michael Munn]{University of Missouri}
\email{munnm@missouri.edu}

\begin{abstract}
In this paper we apply techniques from optimal transport to study the neckpinch examples of Angenent-Knopf which arise through the Ricci flow on $\mathbb{S}^{n+1}$. In particular, we recover their proof of 'single-point pinching' along the flow i.e. the singular set has codimension 1. Using the methods of optimal transportation, we are able to remove the assumption of reflection symmetry for the metric. Our argument relies on the heuristic for weak Ricci flow proposed by McCann-Topping which characterizes super solutions of the Ricci flow by the contractivity of diffusions. 
\end{abstract}

\maketitle

\section{Introduction}
Given a closed Riemannian manifold $(M^n,g)$, a smooth family of Riemannian metrics $g(t)$ on $M^n$ is said to evolve under the Ricci flow \cite{Hamilton1982} provided
\be
\label{RF-eq}
\begin{cases}
	\dfrac{\partial}{\partial t} g(t) = -2\Ric (g(t))\\
	g(0) = g.
\end{cases}
\ee
The uniqueness and short time existence of solutions was shown by Hamilton \cite{Hamilton1982} (see also DeTurck \cite{DeTurck}). Since its introduction, the Ricci flow has become one of the most intensely studied geometric flows in the literature and understanding the formation of singularities of the Ricci flow plays a particularly important role. A finite time singularity occurs at some $T<\infty$ implies 
\[
\lim_{t \nearrow T} \max_{x\in M^n}|\Rm(x,t)| = \infty.
\]
By the maximum principle, it follows that a singularity will develop in finite time once the scalar curvature becomes everywhere positive. A simple example of this can be seen for the canonical round unit sphere $\mathbb{S}^{n+1}$ which collapses to a point along the flow at $T=\frac{1}{2n}$. \\

While the shrinking sphere describes a global singularity, the `neckpinch' examples we are concerned with in this paper are local singularities; i.e.  they occur on a compact subset of the manifold while keeping the volume positive. Intuitively, a manifold shaped like a dumbbell develops a finite-time local singularity as the neck part of the dumbbell contracts. The first rigorous examples of such a singularity were constructed by Angenent-Knopf \cite{AK1} who produced a class of rotationally symmetric initial metrics on $\mathbb{S}^{n+1}$, for $n\geq 2$, which develop local Type-I neckpinch singularities through the Ricci flow (examples for non-compact manifolds did already exist \cite{Simon-Pinch,Feldman-Ilmanen-Knopf}).  In a follow up paper, assuming in addition that the metric is reflection-invariant and that the diameter remains bounded throughout the flow, Angenent-Knopf \cite{AK2} show that the `neckpinch'  singularity occurs precisely at the equator, on the totally geodesic hypersurface $\{0\} \times \sphere^n$. \\

The purpose of this paper is to reframe part of their results using the methods of optimal transportation and provide a separate proof of this single-point pinching which holds without the condition that the metric be reflection invariant. \\ 

Writing $\mathbb{S}^{n+1}$ as a warped product, an $SO(n+1)$-invariant metric on  $(-1,1) \times \sphere^n$ can be written (using a more geometric variable $r$) as
\[
g= dr^2 + \psi^2(r) g_{can},
\]
where $g_{can}$ is the standard round metric on $\sphere^n$ and where $r$ represents the distance from the equator. Note that $\psi$ must satisfy certain boundary conditions to ensure smoothness of the metric at the endpoints. For a time-dependent family of such metrics $g(t)$ which are solutions to the Ricci flow (\ref{RF-eq}), 
the quantity $\psi(r,t) >0$ may be regarded as the radius of the hypersurface $\{r \} \times \sphere^n$ at the time $t$. We show
\begin{thm}
\label{thm-main-1}
Let $g(t)$ be a family of smooth metrics on $\sphere^{n+1}$ which satisfies (\ref{RF-eq}) for $t \in [0,T)$ and which develops a neckpinch singularity at $T < \infty$ as prescribed by Angenent-Knopf (see Section \ref{Neckpinch prelim}). If the diameter $(\sphere^{n+1}, g(t))$ remains  bounded as $ t \nearrow T$, then the singularity occurs only on a totally geodesic hypersurface  of $\{x_0\} \times \sphere^{n}$, for some $x_0 \in (-1,1)$. 
\end{thm}

The original result of Angenent-Knopf assumes reflection symmetry of the metric in addition to  a diameter bound to prove equatorial pinching (see Section 10 of \cite{AK1}). Later, they verify in \cite{AK2} that reflection symmetry in fact implies the diameter bound. Their proof relies on careful analysis and detailed computations arising from the imposed evolution equations on the profile warping function $\psi(r,t)$. In Section \ref{diam estimates} we provide an alternate proof of this `one-point pinching' result. Our method of proof involves  techniques of optimal transportation and, as such, avoids the requirement of reflection symmetry, though we still require the condition on the diameter bound.\\

This paper is organized as follows:  In Section \ref{prelims}, we review relevant background material and set the notation we will use throughout the paper. In particular, we begin in Section \ref{Neckpinch prelim} by discussing the formation of neckpinch singularities through the Ricci flow and the results of Angenent-Knopf mentioned above. In Section \ref{OT prelim} we recall the basic ideas of optimal transportation and in Section \ref{sec-OT-Rot} we study the optimal transport problem in the presence of rotational symmetry. Sections \ref{RF and OT prelim} and \ref{sec-conj heat kernel} are devoted to the metric-measure characterization of diffusions under Ricci the flow and related work of McCann-Topping. In Section \ref{sec-neckpinch control} we give a proof of main theorem and briefly mention how these ideas can be generalized to address more general neckpinch singularities.

\section*{Acknowledgements.}
The authors would like to thank Christina Sormani who originally proposed the study of Ricci flow neckpinch using metric geometry. Also SL is very grateful to Dan Knopf for many insightful conversations and Professor K-T Sturm for giving him the opportunity of working as a postdoctoral fellow at the Hausdorff Center for Mathematics at the University of Bonn.\\

 SL was supported by the Hausdorff Center for Mathematics at the University of Bonn and in part by the National Science Foundation under DMS-0932078 000 while in residence at the Mathematical Science Research Institute in Berkeley,
California, during the Fall of 2013. MM was partly supported by the NSF under the grant OISE-0754379 and thanks the Mathematics Institute at the University of Warwick, where part of this work was done.

\section{Preliminaries}
\label{prelims}
\subsection{Neckpinch Singularities of the Ricci Flow} 
\label{Neckpinch prelim}
Singularities of the Ricci flow can be classified according to how fast they are formed. A solution $(M^n, g(t))$ to (\ref{RF-eq}) develops a Type I, or rapidly forming, singularity at $T < \infty$, if 
\be 
\sup_{M\times [0, T)} (T-t)|\Rm (\cdot, t)| < +\infty.
\ee
Hamilton showed  \cite{Hamilton1982} that such singularities arise for compact 3-manifolds with positive Ricci curvature. Later, B\"{o}hm-Wilking  \cite{BW2008} extended this result to all dimensions showing that any compact manifold with positive curvature operator must develop a Type I singularity in finite time. \\

A solution $(M^{n+1}, g(t))$ of the Ricci flow is said to develop a \textit{neckpinch singularity} at some time $T < \infty$ through the flow by pinching an almost round cylindrical neck. More precisely, there exists a time-dependent family of proper open subsets $N(t) \subset M^{n+1}$ and diffeomorphisms $\phi_t: \R \times \sphere^n \to N(t)$ such that $g(t)$ remains regular on $M^{n+1} \setminus N(t)$ and the pullback $\phi_t^*\left(\left.g(t)\right|_{N(t)}\right)$ on $\R \times \sphere^n$ approaches the ``shrinking cylinder'' soliton metric
\[
ds^2 + 2(n-1)(T-t) g_{can}
\]
in $\mathcal{C}^{\infty}_{loc}$ as $t \nearrow T$, where $g_{can}$ denotes the canonical round metric on the unit sphere $\sphere^n(1) \subset \mathbb{R}^{n+1}$.\\

Following \cite{AK1}, consider $\sphere^{n+1}$ and remove the poles $P_{\pm}$ to identify $\sphere^{n+1} \setminus{P_{\pm}}$ with $(-1,1) \times \sphere^n$. An SO($n+1$)-invariant metric on $\mathbb{S}^{n+1}$ can be written as 
\be 
\label{initial rot sym g} 
g = \phi(x)^2 (dx)^2 + \psi(x)^2 g_{can},
\ee
where $x \in (-1,1)$. Letting $r$ denote the distance to the equator given by 
\be 
\label{r defn} 
r(x) = \int_0^x \phi(y) dy,
\ee
one can rewrite (\ref{initial rot sym g}) more geometrically as a warped product 
\[
g = (dr)^2 + \psi(r)^2 g_{can}.
\]
To ensure smoothness of the metric at the poles, we require $\lim_{x\to \pm 1} \psi_r = \mp 1$ and $\phi/(r_{\pm}-r)$ is a smooth even function of $r_{\pm}-r$, where $r_{\pm} :=r(\pm1)$. This assumption of rotational symmetry on the metric allows for a simplification of the full Ricci flow system from a nonlinear PDE into a scalar parabolic PDE in one space dimension.\\

Throughout, we keep in mind that $r$ depends on $x$ and ultimately, for evolving metrics, both $r$ and $\psi$ will depend on both $x$ and $t$.  In particular, 
\[
\frac{\partial}{\partial r} = \frac{1}{\phi(x)} \frac{\partial }{\partial x}
\]
and $dr = \phi(x) dx$, even when the metric evolves. 
\\

In \cite{AK1}, the authors establish  the existence of Type I neckpinch singularities for an open set of initial $SO(n+1)$-invariant  metrics on $\sphere^{n+1}$ which also have the following properties: 
\begin{enumerate}[(i)]
\item positive scalar curvature 
\item positive sectional curvature on the planes tangential to $\{x\} \times \sphere^n$
\item ``sufficiently pinched'' necks; i.e. the minimum radius should be sufficiently small relative to the maximum radius (c.f. Section 8 of \cite{AK1})
\end{enumerate}
Throughout this paper, let $\AK$ denote this open subset of initial metrics on $\sphere^{n+1}$ and we refer to the neckpinch singularities which arise from these initial metrics as {\em neckpinches of Angenent-Knopf type}. Furthermore, let $\AK_0 \subset \AK$ denote those metrics in $\AK$ which are also reflection symmetric (i.e. $\psi(r, 0) = \psi(-r,0)$).  \\

Briefly summarizing some of the results of Angenent-Knopf found in \cite{AK1,AK2} mentioned above, they show

\begin{thm}
\label{thm-AK12}
{\em (c.f. Section 10 in \cite{AK1}, Lemma 2 in \cite{AK2})}
Given an initial metric $g_0 \in \AK$, the solution $(\sphere^{n+1}, g(t))$ of the Ricci flow becomes singular, developing a neckpinch singularity, at some $T < \infty$. Furthermore, provided $g_0 \in \AK_0$, its diameter remains bounded for all $t \in [0, T)$ and the singularity occurs only on the totally geodesic hypersurface $\{0\} \times \sphere^n$. 
\end{thm}

In particular, to prove ``one-point pinching'', Angenent-Knopf  assume that the initial metric has at least two bumps denoting the locations of those bumps by $x = a(t)$ and $x = b(t)$ for the left and right bump (resp.)  To show that the singularity occurs only on $\{0\} \times \sphere^{n}$, they show that $\psi(r, T) >0$, for $r>0$. Their proof of this one-point pinching requires delicate analysis and construction of a family of subsolutions for $\psi_r$ along the flow.\\

Later, in  \cite{ACK}, Angenent-Caputo-Knopf extended this work by constructing smooth forward evolutions of the Ricci flow starting from initial singular metrics which arise from these rotationally symmetric neck pinches on $\sphere^{n+1}$ described above. To do so requires a careful limiting argument and precise control on the asymptotic profile of the singularity as it emerges from the neckpinch singularity. By passing to the limit of a sequence of these Ricci flows with surgery, effectively the surgery is performed at scale zero. A smooth complete solution $(M^n, g(t))$ on some interval $t \in (T, T')$ of the Ricci flow is called a {\em forward evolution} of a singular Riemannian metric $g(T)$ on $M^n$ if on any open set $\mathcal{O} \subset M^n$, as $t \searrow T$ the metric $g(t)|_{\mathcal{O}}$ converges smoothy to $g(T)|_{\mathcal{O}}$. \\

They prove

\begin{thm}
\label{thm-ACK}
{\em (c.f. Theorem 1 in \cite{ACK})}
Let $g_0$ denote a singular Riemannian metric on $\sphere^{n+1}$, for $n\geq 2$, arising as the limit as $t \nearrow T$ of a rotationally symmetric neckpinch forming at time $T <\infty$. Then there exists a complete smooth forward evolution $(\sphere^{n+1}, g(t))$ for $T < t < T'$ of $g_0$ by the Ricci flow. Any such smooth forward evolution is compact and satisfies a unique asymptotic profile as it emerges from the singularity.
\end{thm}

Viewed together, the work of \cite{AK1, AK2, ACK} and Theorems \ref{thm-AK12} and \ref{thm-ACK} above provide a framework for developing the notion of a ``canonically defined Ricci flow through singularities'' as conjectured by Perelman in \cite{Perelman-RFWS}, albeit in this rather restricted context of non degenerate neckpinch singularities which arise for these particular initial metrics on $\mathbb{S}^{n+1}$ described by Angenent-Knopf (i.e. for $g_0 \in \AK_0$). Such a construction for more general singularities remains a very difficult issue, even for more general neckpinch singularities. Up to this point, continuing a solution of the Ricci flow past a singular time $T < \infty$ required surgery and a series of carefully made choices so that certain crucial estimates remain bounded through the flow \cite{Hamilton-Four-Manifolds, Perelman-RFWS, Perelman-entropy, Perelman-RFsolns}. A complete canonical Ricci flow through singularities would avoid these arbitrary choices and would be broad enough to address all types of singularities that arise in the Ricci flow.\\

\subsection{Optimal Transportation}
\label{OT prelim}
Let $(X,d)$ be a compact metric space and consider the space of Borel probability measures on $X$, denoted $\mathscr{P}(X)$. Given two probability measures $\mu_1, \mu_2 \in \mathscr{P}(X)$, the $p$-Wasserstein distance between them is given by 
\be
\label{W distance}
W_p(\mu_1, \mu_2) = \left[\inf_{\pi \in \Gamma(\mu_1,\mu_2)} \int_{X\times X} d(x,y)^p d\pi(x,y)\right]^{1/p}, \quad \text{ for } p \geq 1,
\ee
where the infimum is taken over the space $\Gamma(\mu_1, \mu_2)$ of all  joint probability measures $\pi$ on $X \times X$ which have marginals $\mu_1$ and $\mu_2$; i.e.  for projections $\pr_1, \pr_2: X \times X \to X$ onto the first and second factors (resp.), one has
\[{\pr_1}_* \pi = \mu_1 \quad \text{ and } \quad  {\pr_2}_* \pi = \mu_2.\]
Any such probability measure $\pi\in \Gamma(\mu_1, \mu_2)$ is called a \emph{transference plan} between $\mu_1$ and $\mu_2$ and when $\pi$ realizes the infimum in (\ref{W distance}) we say $\pi$ is an \emph{optimal transference plan}.\\ 

The Wasserstein distance makes $\mathscr{P}(X)$ into a metric space $(\mathscr{P}(X),W_p)$ which inherits many of the properties of $(X,d)$. In particular, if $(X,d)$ is a complete, separable compact length space then so is  $(\mathscr{P}(X),W_p)$. A constant speed minimizing geodesic in $(\mathscr{P}(X), W_p)$ is called a \emph{Wasserstein geodesic} and it is possible that there are an uncountable number of minimizing Wasserstein geodesics between two given measures $\mu_0 , \mu_1 \in \mathscr{P}(X)$. This is in part caused by some sort of symmetry in the space $X$.\\ 

Denote by $\Gamma(X)$ the set of all minimizing geodesics $ \gamma : [0,1] \to X$ and note that this set is compact in the uniform  topology. A \emph{dynamical transference plan} is a probability measure $\Pi$ on $\Gamma(X)$ such that 
\be
	\pi := (e_0 , e_1)_* \Pi,
\ee
is a transference plan between $\mu_0$ and $\mu_1$, where $e_t$ denotes the evaluation map $e_t : \Gamma(X) \to X$ given by 
\be
	e_t (\gamma) = \gamma (t), \quad \text{ for } t \in [0,1].
\ee

A dynamical transference plan $ \Pi$ is called \emph{optimal} if $\pi$ is an optimal transference plan. If $\Pi$ is an optimal dynamical transference plan, then the displacement one-parameter family of probability measures 
\be
	\mu_t = (e_t)_*  \Pi
\ee
is called a \emph{displacement interpolation}. One can show that any displacement interpolation is a Wasserstein geodesic and, furthermore, any Wasserstein geodesic is the displacement interpolation for some optimal dynamical transference plan (see \cite{Lott-Villani-09}). \\

\subsection{Optimal Transport in Presence of Rotational Symmetry.}
\label{sec-OT-Rot}
In this Section, we examine the optimal transport problem for two rotationally invariant uniform measures $\mu_0, \mu_1$ on a given metric space which also has axial  symmetry. The main result of this section is the rather unsurprising fact that under some conditions, the transport problem reduces to the transport problem on $\R$.\\

In order to simulate symmetry in metric measure spaces, we introduce the notion of a \emph{balanced metric measure space}. This is a rather strong condition and so is symmetry.

\begin{defn}[Balanced Metric Measure Spaces]
For some real number $r_0$, a metric measure space $(X,d_X, m)$ is said to be \emph{balanced at the scale $r_0$} if the function $f_r : X \to \R$ defined by
\be
	f_r (\cdot) := m \big( B\left( \cdot , r  \right)  \big)
\ee
is constant for all $r \in \left(0 , r_0 \right)$; i.e. $x \mapsto m\left( B(\cdot, r)\right)$ does not depend on  $x \in \supp(m) \subset X$, for all $r \in (0, r_0)$.
\end{defn}

Similar definitions for balanced metric measure spaces have been used by Sturm (c.f. Section 8.1 of \cite{Sturm-SpaceSpace}). We also define the concept of axial symmetry for measures on warped products of metric measure spaces. Recall, 

\begin{defn}[c.f. Definition 3.1 in \cite{Chen1999}]
\label{Chen-warped prod}
Suppose  $(X, d_X)$, $(Y, d_Y)$ are two metric spaces and  $f: X \to \R^+$ is a continuous function on $X$. For two points $a,b \in X \times Y$, define
\[
d(a,b) := \inf_{\gamma} \{L_f(\gamma) ~:~ \gamma \text{ is a curve from } x \text{ to } y\},
\]
where $L_f(\gamma)$ denotes the length of the curve $\gamma(s) := (\alpha(s), \beta(s)) \in M \times N$ given by 
\[
L_f(\gamma) : = \lim_{\tau} \sum^n_{i=1} \sqrt{d_X^2(\alpha(t_{i-1},)\alpha(t_i)) + f^2(\alpha(t_{i-1}))d_Y^2(\beta(t_{i-1}), \beta(t_i))}
\]
and the limit is taken with respect to the refinement ordering of partitions $\tau$ of $[0,1]$ denoted by $\tau : 0 = t_0 < t_1 < \cdots < t_n = 1$.
$X\times Y$ equipped with this metric $d$ is called \emph{the warped product of $X$ and $Y$ with warping function $f$} and is denoted $(X\times_f Y, d)$.
\end{defn}

Chen \cite{Chen1999} verifies that $d$ is indeed a distance metric on $M\times N$ and further proves the following properties for geodesics in $(M\times_f N,d)$.

\begin{prop}[see Theorem 4.1 in \cite{Chen1999}]
\label{Chen-geod exists}
With $(M, d_M), (N,d_N)$ and $f$ as above, if $M$ is a complete, locally compact metric space and $N$ is a geodesic space, then for any $a,b \in (M\times_f N, d)$ there exists a geodesic joining $a$ to $b$.
\end{prop}

\begin{prop}[see Lemma 3.2 in \cite{Chen1999}]
Geodesics in $X$ lift horizontally to geodesics in $(X \times_f Y,d)$. Furthermore, if $(\alpha, \beta)$ is a geodesic in $(X \times_f Y, d)$ then $\beta$ is a geodesic in $N$.
\end{prop}

Motivated by these generalizations of warped products to metric spaces, we define 

\begin{defn}[Axially Symmetric Measures]
 Let $(X , d_X, m_X)$ be a metric-measure space and $Y = (\R \times_f X, d)$. Let $\mu \in \mathscr{P}(Y)$ be absolutely continuous with respect to the product measure $dr \times m_X$ on $\R \times X$ and, for $t \in \R$, denote the projection map $\pi_t: Y \to \{t\} \times X$. We say \emph{$\mu$ is axially symmetric at scale $r_0$} if, given any $t \in \R$, the metric-measure space
 \be
 	\Big(\{t\} \times X , f(t)d_X, \left(\pi_t\right)_* \mu \Big)
\ee 
is balanced at scale $r_0 f(t)$. In particular, this means we can write the density function\\ \mbox{$d\mu/ (dr \times dm_X) = \rho(t, x)  : X \to \R$} as a constant function, for all fixed $t \in \R $.
\end{defn}

\begin{defn}[Radial Geodesic]
With $X$ and $Y$ as above and fixing $p_0 \in X$, any geodesic of the form $\gamma(s)= \left( \alpha(s) , p_0 \right): [0,1] \to X$ is called a \emph{radial geodesic}.
\end{defn}

In order to get a nice picture, our model object of study will be two canonical round spheres  joined by an interval of length $L$. We will be interested in studying the optimal transportation of uniform probability measures on balls around the branching points. The reason for considering uniform probability measures will become clear in later sections. 

\begin{thm}\label{thm-rotsym-transport}

Let $Y= (\R \times_f X, d)$ as above and let $\mu_0$ and $\mu_1$ be two axially symmetric probability measures on $Y$ at scales $r_0$ and $r_1$, respectively. Let $\Pi$ be an optimal dynamical transference plan between $\mu_0$ and $\mu_1$ and let 
\be
\Lambda := \left\{ \text{ radial geodesics } \; \gamma: [0,1] \to Y  \right\} \subset \Gamma(Y).
\ee
Then
\be
	\Pi \left( \Lambda \right) = 1.
\ee
Namely, the mass will be transferred almost surely only along the radial geodesics.
\end{thm}

\begin{proof}
	The proof relies on a perturbation argument. Set $\pi : = (e_0, e_1)_\#\Pi$ which by assumption is an optimal transference plan between $\mu_0$ and $\mu_1$. We assume that the conclusion is false; i.e. that $\pi$ does not (almost surely) transport mass along radial geodesics in $Y$, and then find a perturbation $\widetilde{\pi} : = \pi +  \nu$ of $\pi$ which violates the optimality of $\pi$.\\
	
	Choose $t,s \in \R$, ($t \neq s$). Let 
\be
	\A := \left\{ \Big(  \left( t,p \right) ,  \left( s , q  \right) \Big) \in \operatorname{supp}\left( \pi \right) ~:~ p \neq q \right\} \subset Y \times Y,
\ee
and set 
\be
	\Gamma_{\A} := \left\{ \gamma \in \Gamma \left( Y \right) ~:~ \left( e_0(\gamma), e_1(\gamma)  \right) \in \mathcal{A} \right\}.
\ee
By the hypothesis we can assume that $\pi(\A) > 0$ and $\Pi(\Gamma_{\A}) >0$. To show that the perturbation $\widetilde{\pi}$ violates the optimality of $\pi$, it suffices to define $\nu$ in such a way that $\nu \in \mathscr{P}(Y \times Y)$ satisfies the following conditions:

\begin{enumerate}[{\bf (i)}]
\item $\nu$ is a signed measure
\item $\nu^- \le \pi$
\item 	${\pr_i}_* \nu = 0$,  for $i = 1,2$ (i.e. marginals are null)
\item $\int_{Y \times Y} \left| x-y  \right| d\nu(x,y) <0$
\end{enumerate}

Let $\beta>0$ be small enough such that by continuity the $\beta$-neighborhood around $\A$, denoted $\A_{\beta}$ still satisfies
\be
	\pi \left( \mathcal{A}_\beta \right) >0 \quad \text{ and } \quad  \Pi \left( \Gamma_{\mathcal{A}_\beta} \right) >0.
\ee

Let $\delta = \min \left\{|t-s|, r_0 , r_1 , f(s)\frac{\beta}{2}   \right\}$ and take two disjoint neighborhoods of $Y$

\be
U = [s-\delta , s + \delta] \times B_{\left( X,f(s)d_X \right)}(p , \delta) 
\ee
and
\be
V = [s-\delta , s + \delta] \times B_{\left( X,f(s)d_X \right)}(q , \delta).
\ee

Now consider two subsets $N_1 = W \times U$ and $N_2 = W \times V$ of $Y \times Y$ with 
\be
W = [t -\delta   , t + \delta] \times B_{\left(  X , f(t) d_X  \right)}(p . \delta).
\ee 
 
Let $m = \pi \left(  W \times V  \right)$ and define
\be
	\nu := m\left( \pr_{W} , \pr_{U}  \right)_* \pi   -    \left( \pr_{W} , \pr_{V}  \right)_* \pi,
\ee
where in the first term $\pr_W$ is the projection $\pr_W: N_1 \to W$ and $\pr_U: N_1 \to U$ (and $\pr_W, \pr_V$ in the second term are defined similarly on $N_2 = W \times V$).\\

It is easy to see that $\nu$ satisfies conditions \textbf{(i)} - \textbf{(iii)}. Also for $\delta$ small enough there exists an $\epsilon$ such that 
\be
	|x-y| < |x-z| + \epsilon, \quad  \text{ for all }  \left( x,y,z  \right) \in N_1 \times N_3 \times N_2,
\ee
where $N_3 = U \times V$. Therefore,
\be
	\int_{Y \times Y} \left| x - y  \right| d\nu(x,y) < \epsilon m <0,
\ee
and hence condition \textbf{(iv)}.
\end{proof}

It follows that for two axially symmetric probability measures on some $\mathbb{R} \times_f  X$ as described above, any mass is transported along radial geodesics. Thus, the optimal transportation problem on $\mathbb{R} \times_f  X$ reduces to the optimal transportation problem on $\mathbb{R}$ alone.

With this in mind, we now recall some well-known facts about optimal transportation on $\R$.

\subsection{Optimal Transport on $\mathbb{R}$.} In this Section, we mention some relevant results in \cite{Villani-TOT} concerning the optimal transportation on the real line. Note that any probability measure $\mu$ on $\mathbb{R}$ can be represented by its cumulative distribution function. Namely,
\[
F(x) = \int_{-\infty}^x d\mu = \mu\left[ (-\infty, x)\right].
\]
From basic probability theory, it follows that $F(x)$ is right-continuous, nondecreasing, $F(-\infty)=0$, and $F(+\infty)=1$. Furthermore, one defines the generalized inverse of $F$ on $[0,1]$ as
\[F^{-1}(t) = \inf\{x \in \mathbb{R} ~:~ F(x) >t\}.\]

\begin{prop}[Quadratic Cost on $\mathbb{R}$, see Theorem 2.18 in \cite{Villani-TOT}]
\label{prop-real-line-transport}
For two probability measures $\mu_1$ and $\mu_2$ on $\R$ with respective cumulative distributions $F$ and $G$, we consider $\pi$ to be the probability measure on $\R^2$ with joint two-dimensional cumulative distribution function
\be
	\mathbf{H}(x,y) = \min \left\{ F(x) , G(y) \right\}.
\ee
Then $\pi \in \Gamma(\mu_1, \mu_2)$ and 
\[ 
W_2^2(\mu_1,\mu_2) = \int_0^1 |F^{-1}(t) - G ^{-1}(t)|^2 dt.
\]
\end{prop}

Furthermore, by the Hoeffding-Fr\'{e}chet Theorem a nonnegative function $H$ on $\mathbb{R}^2$ which is nondecreasing and right continuous in each argument, defines a probability measure $\pi$ on $\mathbb{R}^2$ with marginals $\mu_1$ and $\mu_2$ if and only if
\be
	F(x) + G(y) -1 \le \mathbf{H}(x,y) \le \min \left\{ F(x) , G(y) \right\}, \quad \text{ for all } (x,y) \in \R^2.
\ee

In fact, Proposition \ref{prop-real-line-transport} remains true when replacing $d(x,y)^p$ with any convex cost function. \\

In particular, taking $p=1$, it follows from Fubini's theorem that 
\be
\label{W1 vs L1 dist} 
W_1(\mu_1,\mu_2) = \int_0^1 |F^{-1}(t) - G ^{-1}(t)| dt=  \int_{\R} |F(x) - G (x)| dx.
\ee

\subsection{Ricci Flow and Optimal Transport} 
\label{RF and OT prelim}
Given a solution to (\ref{RF-eq}) on some time interval $t \in [0, T]$, as in \cite{Perelman-entropy}, let $\tau := T-t$ denote the backwards time  parameter. Note that  $\del/ \del \tau = - \del / \del t$. The family of metrics $g(\tau)$ is said to satisfy the \emph{backwards} Ricci flow, \be
\label{bRF-eq}
\frac{\partial g}{\partial \tau} = 2\Ric(g(\tau)).
\ee

A family of measures  $\mu(\tau)$ is called a diffusion if $d\mu(\cdot, \tau) := u(\cdot, \tau) \dvol_{g(\tau)}$ and $u(\cdot, \tau)$ satisfies the conjugate heat equation 
\[
\frac{\partial u}{ \partial \tau } =  \Delta_{g(\tau)} u - \left(\frac{1}{2} \tr \frac{\partial g}{\partial \tau}\right) u.
\]
Note that since the Riemannian volume form $\dvol_{g(\tau)}$ evolves according to
\be
\label{vol evolve}
\frac{\partial}{\partial \tau}\dvol_{g(\tau)} = \frac{1}{2}\left(\tr \frac{\partial g}{\partial \tau}\right)\dvol_{g(\tau)},
\ee
the additional term in the conjugate heat equation is needed to ensure that the measures $\nu(\tau)$ remain probability measures throughout the flow. Taking into account (\ref{bRF-eq}), we can then write the conjugate heat equation as  
\be
\label{conj heat eqn}
\frac{\partial u}{ \partial \tau } =  \Delta_{g(\tau)} u - R_{g(\tau)}u.
\ee
Here $R_{g(\tau)} = \tr \Ric(g(\tau))$ denotes the scalar curvature. We denote the \emph{conjugate heat operator} acting on functions $u: M \times [0,T] \to \mathbb{R}$ by 
\be
\label{conj heat operator}
\square^* u := \frac{\del }{\del \tau}u - \Delta_{g(\tau)}u + R_{g(\tau)}u.
\ee

In \cite{McCann-Topping}, McCann-Topping prove contractivity  of the Wasserstein distance between two diffusions on a manifold evolving by the {\it backwards} Ricci flow. Note that, in this setting, the Wasserstein distance is measured with respect to the changing metric $g(\tau)$. That is,
\be
W_p(\mu_1(\tau), \mu_2(\tau), \tau) := \left[\inf_{\pi \in \Gamma(\mu_1,\mu_2)} \int_{X\times X} d_{g(\tau)}(x,y)^p d\pi(x,y)\right]^{1/p}, \quad \text{ for } p \geq 1,
\ee
where $d_{g(\tau)}$ denotes the distance induced by the Riemannian metric $g(\tau)$; compare with (\ref{W distance}). They show 

\begin{thm}[c.f. \cite{ToppingOTNotes}]
\label{thm-McCannTopping}
	Given a compact oriented manifold $M$ equipped with a smooth family of metrics $g(\tau)$, for $\tau \in [\tau_1 , \tau_2]$, the following are equivalent
\begin{enumerate}[(i)]
\item	 $g(\tau)$ is a super-solution to the Ricci flow (parameterized backwards in time); i.e.\\
	 \[\frac{\partial g}{\partial \tau} \le 2 \Ric (g(\tau)).\]
	
\item For $\tau_1 < a < b < \tau_2$ and two diffusions $\mu_1(x,\tau)$ and $\mu_1(x,\tau)$ on $M$, the distances 
	\be
		W_1 \left( \mu_1(\tau) , \mu_2(\tau), \tau \right) \quad \text{ and } \quad W_2 \left( \mu_1(\tau) , \mu_2 (\tau), \tau \right)
	\ee
are nonincreasing functions of $\tau \in (a,b)$. \\

\item	For $\tau_1 < a < b < \tau_2$, and $f : M \times (a,b) \to \R$ which solves  $ - \frac{\partial f}{\partial \tau} = \Delta_{g(\tau)} f$, the Lipschitz constant of $f(\cdot, \tau)$ with respect to $g(\tau)$ given by
	\be
		\Lip (f(\cdot,\tau)) := \sup_M |\nabla f(\cdot,\tau)|,
	\ee
is nondecreasing on $(a,b)$.
\end{enumerate}	
\end{thm}

Theorem \ref{thm-McCannTopping} gives a characterization of super-solutions to the Ricci flow (parameterized backwards in time) via contractivity of certain Wasserstein distance between diffusions on the evolving manifold. Thus, their work provides one direction in which to explore definitions of the Ricci flow for more general metric measure spaces. We adopt this heuristic in our argument. However, to do so, requires a more general notion of diffusions and a definition of the conjugate heat operator (\ref{conj heat operator}) which uses only the  metric-measure properties of the space. We do this in the following section.
 
\subsection{Constructing the Conjugate Heat Kernel}
\label{sec-conj heat kernel}
Following the work of von Renesse-Sturm \cite{Sturm-Von}, given a family of metrics $g(\tau)$ satisfying (\ref{bRF-eq}), in this subsection, we show how to define the conjugate heat operator (\ref{conj heat operator}) as a limit of operators defined using the metric-measure properties of the space. \\

In Section 4.1 of \cite{Sturm-Von}, von Renesse-Sturm show how the heat kernel of a fixed Riemannian manifold $(M,g)$ may be obtained by applying the Trotter-Chernov product formula to a limit of a family of Markov operators which are defined using only metric-measure properties of $M$. Namely, define a family of Markov operators $\sigma_r$ acting on the set $\mathcal{F}_B$ of bounded Borel measurable functions  by $\sigma_rf(x) = \int_M f(y)~ d\sigma_{r,x}(y)$, where the measure $\sigma_{r,x}$  is given by
\be 
\label{sigma-rx} 
\sigma_{r,x} (A) := \frac{\Haus^{n-1}(A \cap \partial B(x,r))}{\Haus^{n-1}(\partial B(x,r))}, \quad \text{ for Borel measurable sets } A \subset M.
\ee
Generalizing this argument to a time-dependent family of operators, in \cite{Lakzian-Munn} we show how to approximate $\Delta_{g(\tau)}$ relying  only  on the metric-measure properties of the space. \\

Namely, suppose $g(\tau)$ is a family of metrics on $M$ satisfying (\ref{bRF-eq}) for $\tau \in [0, T]$, $0<T<\infty$. As in (\ref{sigma-rx}), define the normalized Riemannian uniform distribution on spheres centered at $x \in (M, g(\tau))$ of radius $r>0$ by 
\be 
\label{sigmatrx} 
\sigma^{\tau}_{r,x} (A) := \frac{\Haus^{n-1}(A \cap \partial B^{\tau}(x,r))}{\Haus^{n-1}(\partial B^{\tau}(x,r))},
\ee
where $B^{\tau}(x,r)$ denotes the ball of radius $r$ centered at $x$ with respect to the metric $g(\tau)$. As before, setting $\sigma^\tau_rf(x) : =  \int_M f(y)~ d\sigma^{\tau}_{r,x}(y)$  it follows that for $f \in C^3((M, g(\tau))$, 
\be
\label{sigmataur}
\sigma^\tau_rf(x) = f(x) + \frac{r^2}{2n}\Delta_{g(\tau)}f(x) + o(r^2).
\ee

To obtain a similar description of the scalar curvature, recall that on a Riemannian manifold $(M,g)$, \cite[Theorem 3.1]{Gray} 
\be
	\frac{\mathcal{H}^{n-1} \left(\partial B^{\tau}(x,r) \subset M^n \right)}{\mathcal{H}^{n-1} \left(\partial B(0,r) \subset \R^n \right)} = 1 - \frac{r^2}{6n} R_{g(\tau)}(x) + o(r^4),
\ee
where as before, $R_g$ denotes the scalar curvature with respect to $g$. Define a family of $\tau$-dependent operators $\theta^{\tau}_r : \mathcal{F}_B \to \mathcal{F}_B$ by
\be
\label{thetataur}
	\theta^\tau_r f(x) := \frac{\mathcal{H}^{n-1} \left(\partial B^{\tau}(x,r) \subset M^n \right)}{\mathcal{H}^{n-1} \left(\partial B(0,r) \subset \R^n \right)} \Id\left( f(x)\right) = f(x) - \frac{r^2}{6n} R_{g(\tau)}(x) f(x) + o(r^4).
\ee

Combining (\ref{sigmataur}) and (\ref{thetataur}), we define a new operator $A^{\tau}_r: \mathcal{F}_B \to \mathcal{F}_B$ by
\be
	A^\tau_r := \frac{1}{4} \sigma^\tau_r + \frac{3}{4} \theta^\tau_r.
\ee
Note that for any $f \in \mathcal{F}_B$,
\be
	A^\tau_r f(x) = f(x) + \frac{r^2}{8n} \left( \Delta_{g(\tau)} - R_{g(\tau)}(x)    \right) f(x) + o(r^2).
\ee

For $\tau \neq 0$, let $F_\tau(t)$ be the strongly continuous contraction semigroup on $\mathcal{F}_B$ generated by \mbox{$(\Delta_{g(\tau)} - R_{g(\tau)}\Id)$}. By applying the classic Trotter-Chernov product formula, one gets for any $t \geq0$ 
\be
\label{Ftau}
	\left( A^\tau_{\sqrt{8nt/j}} \right)^j f(x) \xrightarrow{\quad j \to \infty \quad }  e^{t \left(\Delta_{g(\tau)} - R_{g(\tau)}\Id \right)} f(x).
\ee

Letting $\mathcal{L}(\mathcal{F}_B)$ denote the algebra of all bounded linear operators defined on $\mathcal{F}_B$, we then have, for each $\tau \in [0, T]$, functions $F_{\tau}: [0, +\infty) \to \mathcal{L}(\mathcal{F}_B)$, given by
\be
\label{Ftaut}
F_{\tau}(t) := e^{t(\Delta_{g(\tau)}-R_{g(\tau)}\Id)}.
\ee
Ultimately our goal is to use the operators $F_{\tau}(t)$ to describe solutions to the conjugate heat equation (\ref{conj heat eqn}). Suppose the function $u: M \times [0,T] \to \mathbb{R}$ solves the  initial value problem
\be \label{conj heat IVP}
\begin{cases}
\dfrac{d}{d\tau}u(x,\tau) \!\!\!&=~ \Delta_{g(\tau)}u(x,\tau) - R_{g(\tau)}u(x,\tau) \\
\quad u(x,0) \!&=~ f(x),
\end{cases}
\ee

and that we can write $u(x,\tau)$ as $U(\tau)f(x)$ for some family of bounded linear operators $\{U(\tau)\}_{0\leq \tau \leq T}$ which satisfy the usual strong continuity properties and composition laws; i.e.
\[\lim_{\tau \downarrow 0} U(\tau) = \Id \quad \text{ and } \quad U(\tau_2) \circ U(\tau_1) = U(\tau_1 + \tau_2).\] 

It follows from Vuillermot's generalized Trotter-Chernov product formula for time-dependent operators \cite{Vuillermot}  that 
\begin{lem}
Let $g(\tau)$ be a family of metrics on $M$ which satisfies (\ref{bRF-eq}) for $\tau \in [0,T]$. With $F_{\tau}(t)$ defined as in (\ref{Ftaut}), for any Borel measurable function $f \in \mathcal{B}(M)$, the solution $u(x,\tau)$ to (\ref{conj heat IVP}) is given by
\be
\label{IVP soln}
u(x,\tau) := U(\tau) f(x) = \lim_{m\to \infty}\prod^0_{i=m-1} F_{\frac{i}{m}\tau} \left(\frac{\tau}{m}\right) f(x).
\ee
and the limit converges in the strong operator topology of $\mathcal{L}(B)$\end{lem}

\begin{proof}
To begin, note that $F_{\tau}(0) = \Id$ for every $\tau \in [0, T]$ and $t \mapsto F_{\tau}(t)$ is continuous on $[0,T]$ in the strong operator topology of $\mathcal{L}(\mathcal{F}_B)$. Furthermore, since both $t\Delta_{g(\tau)}$ and $t\Id$ are $C_0$ semigroups, there exists constants $c \geq0$ and $M \geq 1$ such that 
\[
\sup_{\tau \in [0,T]} \left|\left|F_{\tau}(t)\right|\right|_{\infty} \leq Me^{ct}.
\]
As in \cite{Vuillermot}, define the strong limit
\[
F'_{\tau}(0)f :=\lim_{t \to 0^+} \frac{F_{\tau}(t)f(x) -f(x)}{\tau}.
\] 
Note that $F'_{\tau}(0) = (\Delta_{g(\tau)}-R_{g(\tau)}\Id)$ and let $\D(F'_{\tau}(0))$ denote the linear set of all $f \in \mathcal{B}(M)$   such that $F'_{\tau}(0)f$ exists for every $\tau \in [0,T]$.\\

 Furthermore, $U(\tau) f \in \D(F'_{\tau}(0))$ for all $\tau \in (0, T]$, and for every $f \in \mathcal{F}_B$. Also,
 \[
 \lim_{t \to 0^+}\sup_{0 < \tau <T} \left|\left|\frac{F_{\tau}(t) U(\tau) f - U(\tau) f}{\tau} - F'_{\tau}(0)U(\tau)f \right|\right| = 0
 \]

Thus, it follows from \cite{Vuillermot} we have the limit (\ref{IVP soln}) and the limit converges in the strong operator topology on $\mathcal{F}_B$. 
\end{proof}

Rewriting $F_{\tau}(t)$ using (\ref{Ftau}) we can write (\ref{IVP soln}) as 
\begin{eqnarray*}
U(\tau) 	&=& \lim_{m\to \infty}\prod^0_{i=m-1} F_{\frac{i}{m}\tau} \left(\frac{\tau}{m}\right)\\
		&=& \lim_{m\to \infty} \lim_{j\to \infty} \prod^0_{i=m-1}\left( A^{\frac{i}{m}\tau}_{\sqrt{\frac{8n\tau}{jm}}} \right).
\end{eqnarray*}
Motivated by this, we define the conjugate heat equation in the following way. Note, the definition requires only the metric-measure properties of the space. We have

\begin{defn}
\label{weakdiffusion}
Suppose $g(\tau)$ is a family of pseudo-metrics on $M$. We call a family of measures $\mu(\tau)$ a {\bf weak diffusion} if $d\mu(\cdot, \tau) := u(\cdot, \tau) \dvol_{g(\tau)}$ where 
\[
u(x,\tau) = \lim_{m\to \infty} \lim_{j\to \infty} \prod^0_{i=m-1}\left( A^{\frac{i}{m}\tau}_{\sqrt{\frac{8n\tau}{jm}}} \right) u(x,0).
\]
\end{defn}

In the next section, we will use this characterization of the conjugate heat equation to address diffusion of measures evolving through the Ricci flow on a sphere $\mathbb{S}^{n+1}$ which develops a neckpinch singularity of positive length. In short, we show that, under these circumstances, the diameter can not stay bounded while also assuming the contractility of diffusion measures in the Wasserstein distance, as indicated by McCann-Topping.

\section{Proof of Theorem \ref{thm-main-1}}
\label{sec-neckpinch control}
As we saw in Section \ref{sec-OT-Rot}, the optimal transport of rotationally symmetric probability measures in a rotationally symmetric setting reduces (by considering only the radial geodesics) to the optimal transport problem on $\R$ for which we have a rich theory.  In this Section we show how this fact can be used to estimate the Wasserstein distance between two such measures and ultimately prove our Theorem \ref{thm-main-1}.\\ 

\subsection{Framework.} 
\label{framework}
To begin, we describe the framework for our argument. Suppose $\sphere^{n+1}$ is equipped with an initial metric $g_0 \in \AK$, i.e. not necessarily reflection symmetric. It follows from \cite{AK1} that a Type 1 neckpinch singularity develops through the Ricci flow at some finite time $T < \infty$. Ultimately, we aim to prove that the singularity which develops occurs at a single point. By way of contradiction, we assume instead that the neckpinch develops on some interval of positive length. That is, using the notation from Section \ref{Neckpinch prelim},  we assume
\[
\lim_{t \nearrow T} \Rm(x,t) = \infty, \quad \forall ~x \in (x_1, x_2) \times \sphere^n, \text{ where } -1 < x_1 < x_2 <1.
\]
Set 
\[
r_i = r(x_i), \text{ for } i =1,2,
\]
and; recall, $r_{\pm} := r(\pm 1)$ where $r$ is defined by (\ref{r defn}).

At the singular time $t = T$, we view $(\sphere^{n+1}, g(T))$ as a metric space; or rather, a union of metric spaces (see Figure 1). \\

\begin{figure}[h!]
\label{fig}
\centering
\includegraphics[width=4.75in, height=4.5in]{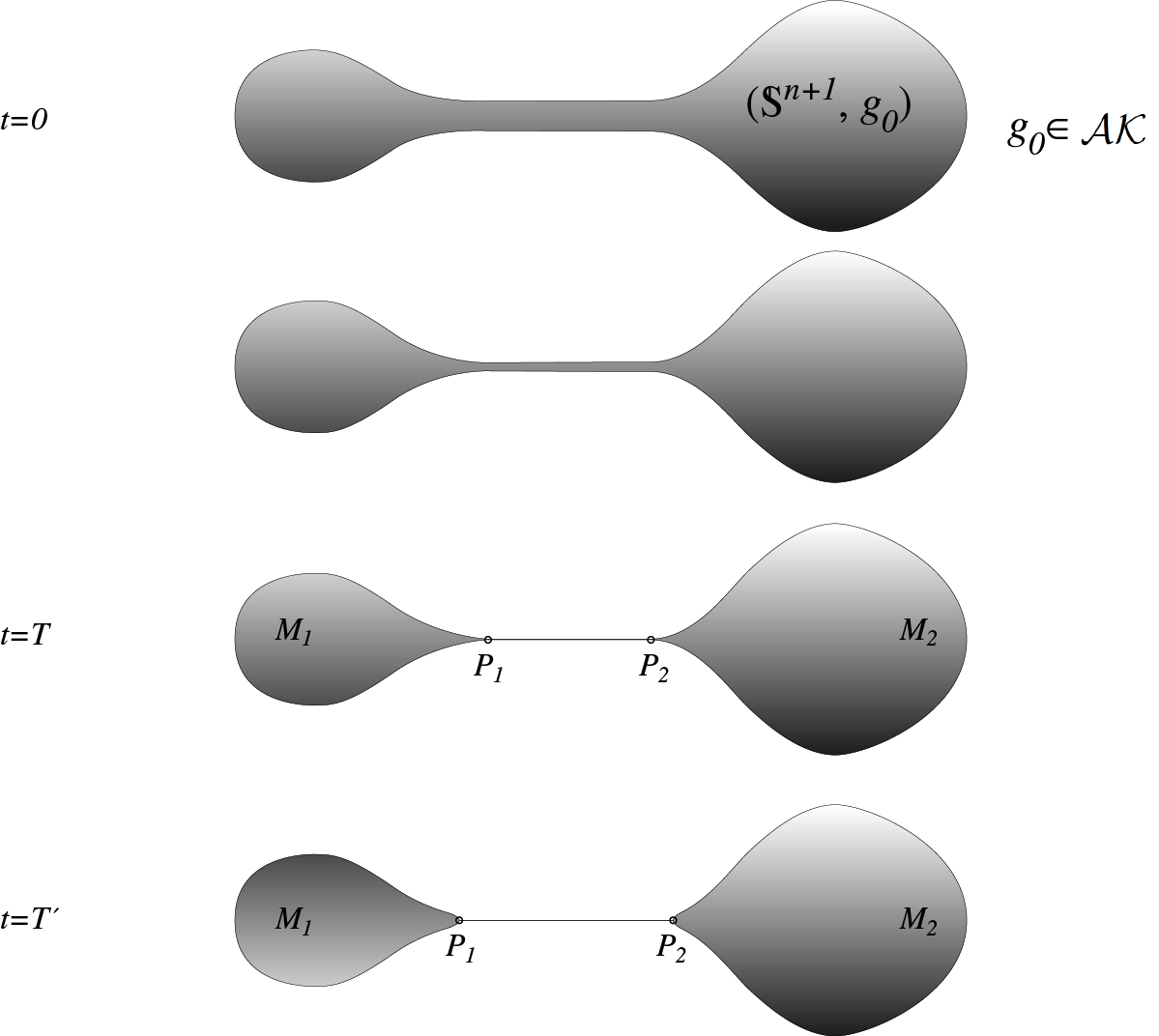}\\
  \caption{Development of neckpinch singularity.}
  \end{figure}

Let $(M_1, g_1)$ denote $(-1, x_1) \times \sphere^n$ with the metric $g_1 = dr^2 + \psi_1(r, T) g_{can}$, where $\psi_ 1= \left.\psi(r,T) \right|_{r\in(r_-, r_1)}$. That is, $g_1$ is the singular metric which arises through the Ricci flow as the limit of the metric $\left. g(t) \right|_{(r_-, r_1) \times \sphere^n}$ as $t \nearrow T$. Similarly, take  $(M_2, g_2)$ to be $(x_2, 1) \times \sphere^n$ with the metric $g_2 = dr^2 + \psi_2(r, T) g_{can}$, where $\psi_2 = \left.\psi(r,T) \right|_{r\in(r_2, r_+)}$. Note that while the metrics $g_1$ and $g_2$ can be extended smoothly to the boundaries $x = \pm 1$, they are {\it not} smooth for $x = x_1$ and $x=x_2$ (see \cite{AK1}). Label these points, which are now the new poles of the degenerate `spheres' $M_1$ and $M_2$ respectively, by  
\begin{align*}
P_1 &= \{x_1\} \times \sphere^n,\\
P_2 &= \{x_2\} \times \sphere^n.
\end{align*}

Laslty, let $[0,L]$ denote the interval of length $L := x_2 - x_1 >0$ with the usual distance metric. We can now consider $(\sphere^{n+1}, g(T))$ as the union
\[
\left( M_1 \sqcup M_2\right) \cup_h [0, L],
\]
which is obtained by first taking the disjoint union $\left( M_1 \sqcup M_2\right)$ and then identifying the boundary of $[0,L]$ to points in $\MM$ via a map $h: \{\{0\}, \{L\}\} \to\left( M_1 \sqcup M_2\right)$ where
\begin{eqnarray*}
h(0) &=& P_1 \in M_1 \\
h(L) &=& P_2\in  M_2.
 \end{eqnarray*} 

Using the work of Angenent-Caputo-Knopf  \cite{ACK}, it is possible to describe a smooth forward evolution of the metrics on the components $M_1$ and $M_2$ beyond the singular time. Employing Theorem \ref{thm-ACK}, taking $g_0$ to be the initial singular metric $g_1(T)$ on $M_1$, it follows that there exists a complete smooth forward evolution $(M_1, g_1(t))$ for $t \in (T, T_1)$ of $g_1(T)$ by the Ricci flow. Similarly, there also exists a complete smooth forward evolution $(M_2, g_2(t))$ for $t \in (T, T_2)$ of $g_2(T)$ on $M_2$. We set  $T' = \min\{T+T_1, T+T_2\}$.\\

In a similar way, for $t >T$, let $[0, L(t)]$ denote the interval of length $L(t)$ which joins $(M_1, g_1(t))$ and $(M_2, g_2(t))$. Note $L(T) = L$; and, for $T \leq t < T'$, the distance between two points  $x \in M_1$ and $y \in M_2$ is given by 
\[
d_{g_1(t)}(x,P_1) + L(t) + d_{g_2(t)}(P_2, y),
\]
where $d_{g_i(t)}$ denotes the distance metric on $M_i$ induced by $g_i(t)$. More generally, we can define a one-parameter family of distance metrics $D(t)$, for $t \in (T, T')$, on the connected metric space given, as before, by 
\[X:= (\MM) \cup_h [0,1].\]
Namely, take the distance between points $(x,y) \in X\times X$ to be  (using $dx$ to denote the usual distance metric on $I = [0,1]$)
\be
\label{D metric}
D(t)(x,y) :=
\begin{cases}
d_{g_1(t)}(x,y), 			& (x,y) \in M_1 \times M_1\\
d_{g_2(t)} (x,y),			& (x,y) \in M_2 \times M_2\\
| x - y|,				& (x,y) \in I \times I\\
d_{g_1(t)}(x, P_1) + L(t) + d_{g_2(t)}(y, P_2), 	& (x,y) \in M_1 \times M_2\\	
d_{g_i(t)}(x,P_i) + L(t) + |P_i - y|,			& (x,y) \in M_i \times I, ~i = 1,2.
 \end{cases}
\ee

Note that, defined this way, 
\[
\lim_{t \searrow T}~ d_{GH}\left( (X,D(t)), (\sphere^{n+1}, d_{g(T)})\right) \to 0.
\]

Also, by \cite{ACK}, as $t \searrow T$, the metrics $g_i(t)$ on $M_i$ are smooth and converge smoothly to $g(T)|_{M_i}$ on  open sets in $M_i \subset \sphere^{n+1}$. 
In other words, we have a one-parameter family of pseudo-metrics $g(t)$ on $\sphere^{n+1}$ defined for $t \in [0, T')$ which are in fact smooth Riemannian metrics on all of $\sphere^{n+1}$ for $t \in [0, T)$ as well as on the open sets $(-1,x_1) \times \sphere^n$ and $(x_2, 1) \times \sphere^n$ for $t \in (T, T')$. \\

To complete the argument, we ultimately examine the behavior of diffusions on $\sphere^{n+1}$ for these pseudo-metrics $g(t)$ for $t \in [0,T')$. To do so requires the metric characterization of diffusions we developed in Section \ref{sec-conj heat kernel}. To summarize, in Section \ref{diam estimates}, we show with Proposition \ref{diambound} that for certain diffusions on $\sphere^{n+1}$ the change in the cumulative distribution function can be made arbitrarily large. Using this fact, we can bound the change in the Wasserstein distance between between two such diffusions. In Section \ref{proof} we combine these facts to prove Theorem \ref{thm-main-1}.\\

Essential to our argument is that, in the construction above, diffusions on $X$ do not transport mass across the interval joining $M_1$ and $M_2$. To verify this, set $\tau := T' -t$ and let $\mu(\tau)$ be a (weak) diffusion on $(\sphere^{n+1}, g(\tau))$, for $\tau \in (0, T']$. Write $d\mu(x,\tau) = u(x,\tau) \dvol_{g(\tau)}(x)$ and note that, for $\tau \in (T, T']$, $\mu(\tau)$ is a smooth diffusion and $u(x,\tau)$ satisfies the conjugate heat equation (\ref{conj heat eqn}). For $\tau \in (0, T'-T]$, $\mu(\tau)$ is defined weakly and $u(x,\tau)$ is as in Definition \ref{weakdiffusion}. Naturally, these notions coincide when $g(\tau)$ is a smooth Riemannian metric, i.e. for $\tau \in (T'-T, T']$. 

\begin{lem}
\label{supp lemma}
For $i=1,2$, if $\supp(\mu(\tau_0)) \subset M_i$, for some $\tau_0 \in (0, T'-T]$; then, $\supp(\mu(\tau)) \subset M_i$, for all $\tau < \tau_0$.
\end{lem}

\begin{proof}
The lemma follows from examining the density function $u(x,\tau)$ on $X$ for $\tau < \tau_0$.\\

Let $x \in (x_1 , x_2) \times \sphere^n$, we have by assumption $u(x, \tau_0) = 0$. Without loss of generality suppose $D(\tau_0)(P_1, x) \leq D(\tau_0)(P_2, x)$. For any $r$ sufficiently small, namely $0 < r < D(\tau_0)(P_1, x)$, it follows from (\ref{sigmataur}) and (\ref{thetataur}) that $\sigma^{\tau_0}_r = \theta^{\tau_0}_r = 0$ as operators on $\mathcal{F}_B$. Thus, by definition, $A^{\tau_0}_r f(x) \equiv 0$, for any $f\in\mathcal{F}_B$.\

Therefore, it follows that 
\[
\lim_{j\to \infty} \left( A^{\tau_0}_{\sqrt{8nt/j}}\right)^j f(x) = 0. 
\]
By definition, $F_{\tau_0}(t) = 0 \in \mathcal{L}(\mathcal{F}_B)$ and therefore, $\left.\dfrac{\partial u}{\partial \tau}\right|_{\tau= \tau_0} = 0$. Furthermore, by Definition \ref{weakdiffusion} we have $u(x, \tau) = 0$ for all $\tau < \tau_0$ as well. Thus, mass is not transported across the interval and $\supp(\mu(\tau)) \subset M_i$ for all $\tau < \tau_0$.
\end{proof}

\subsection{Estimating the Diameter via Wasserstein Distance}
\label{diam estimates}
Let $g(\tau)$ be a family of metrics on $\sphere^{n+1}$ evolving by the backwards Ricci flow equation (\ref{bRF-eq}) and which develops a neck pinch singularity as discussed above. In this section, we show that there exists a diffusion $\nu(\tau)$ for which  rate of change of the Wasserstein distance $\frac{\partial}{\partial \tau} \int_{0}^{\diam(M_1)} d\nu(\tau)$ can be made as large as we want by taking diffusion that are close to Dirac $\delta$ distributions on $\sphere^{n+1}$. For now, we assume that the nonnegativity of the scalar curvature on $[\tau_1 , \tau_2]$ (which is not very restricting as we will see later.)

\begin{prop}
\label{diambound}
Let $(\sphere^{n+1}, g(\tau))$ evolve by (\ref{bRF-eq}) and suppose $R(\tau) \geq 0$ for all $\tau$. For any real number $M>0$ there exist diffusion $\nu(\tau)$ for which
\be 
\frac{\partial}{\partial \tau} \int_{0}^{\diam(M_1)}  F(r,\tau) \dr > M
\ee

where $F(r,\tau)$ 
is the commulative distributions of $\nu(\tau)$. 
\end{prop}

Before we begin the proof, we recall some notation and computations done by Angenent-Knopf \cite{AK1} in \cite{AK1, AK2}. As mentioned in Section \ref{Neckpinch prelim}, setting $r$ to denote the distance from the equator, the Ricci tensor of $g$ in these geometric coordinates is given by
\be
\label{Ricci}
\Ric = -n\frac{\psi_{rr}}{\psi} dr^2 + \left(-\psi\psi_{rr} - (n-1)\psi_r^2 + n-1 \right) g_{can}.
\ee
Furthermore, it follows that the evolution a one-parameter family of metrics $g(t)$ according to the Ricci flow (\ref{RF-eq}) is equivalent to the coupled system
\be
\label{phipsi evolve}
\begin{cases}
\psi_t = \psi_{rr} - (n-1)\frac{1-\psi_r^2}{\psi} \\
\phi_t = \frac{n\psi_{rr}}{\psi}\phi,
\end{cases}
\ee
where the derivative with respect to the parameter $r$ is given by $\frac{\partial}{\partial r}= \frac{1}{\phi(x,t)} \frac{\partial}{\partial x}$ and $dr = \phi(x) dx$. \\

{\it Proof of Proposition \ref{diambound}.}
Let $\nu(\tau) = u(r, \tau) \mathrm{d}vol $ be a diffusion on $(M, g(\tau))$ and thus satisfies the conjugate heat equation (\ref{conj heat eqn}).  Computing from the pole $P_+ \in \sphere^{n+1}$, the cumulative distribution function of $\nu(\tau)$ in terms of $x$ is:
\be
\label{F defn}
		F(x,\tau) =\int_{S^n} \int_0^x \; u(y , \tau) \dvol_{(y,\tau)}. 
\ee
By the divergence theorem, and keeping in mind (\ref{vol evolve}),
\begin{align}
	\frac{\partial}{\partial \tau} F(x, \tau) 	&= \int_{S^n}  \int_0^x \; \left(\frac{\partial u}{\partial \tau}+ R u\right) \dvol , \\ 
								&= \int_{S^n}  \int_0^x \; \Delta u \dvol, \\ 
								&= \int_{\{x\} \times S^n} \; \left< \nabla u , \bf{n} \right> \dsigma ,
\end{align}
where $\bf{n}$ denotes the outward unit normal to the hypersurface $\left\{x \right\} \times \sphere^n$ and $\mathrm{d}\sigma$ its area form. Since the metrics $g(x, \tau)$ and the solution $u(x,\tau)$ are rotationally symmetric, we have $\n = \frac{\partial}{\partial r}= \frac{1}{\phi(x,t)} \frac{\partial}{\partial x}$ and $\nabla u = u_r \frac{\partial}{\partial r} = \frac{u_x}{\phi^2(x)} \frac{\partial}{\partial x}$. Hence, 
\be
	\frac{\partial}{\partial \tau} F(x, \tau) =\int_{\left\{ x \right\} \times \sphere^n} \; \frac{u_x}{\phi^3(x)}\left| \frac{\partial}{\partial x} \right|^2 \dsigma = \int_{\left\{ x \right\} \times \sphere^n} \; \frac{u_x}{\phi(x)} \dsigma.
\ee
Which, in polar coordinates, gives
\be
\label{dFdtau}
	\frac{\partial}{\partial \tau} F(r, \tau) =\int_{\left\{ r \right\} \times S^n}  u_r  \dsigma.
\ee

Using these identities and the evolution equation for $\phi$ given by (\ref{phipsi evolve}). Setting $r_{max} = r(1) = \int_0^1 \phi(y) dy$, we compute
\begin{align}
	\frac{\partial}{\partial \tau} \int_{0}^{r_{max}}  F(r, \tau) \dr 
			&= \frac{\partial}{\partial \tau} \int_0^1  F(x,\tau)\phi(x, \tau) \; \mathrm{d}x   \\ 
			&=  \int_0^1  \;  \left( \frac{\partial  F}{\partial \tau} \phi  +  F\frac{\partial \phi}{\partial \tau} \right)  \dx \\ 
			& =   \int_0^1  \; \left( \frac{\partial  F}{\partial \tau}  -   n F\frac{\psi_{rr}}{ \psi} \right)\; \phi(x,\tau) \dx \\ 
			& = \int_{0}^{r_{max}} \left(  \frac{\partial  F}{\partial \tau}  -   n F\frac{\psi_{rr}}{ \psi} \right) \dr.
\end{align}
Using (\ref{dFdtau}) and (\ref{Ricci}), this can be written as
\begin{align}
	\frac{\partial}{\partial \tau} \int_{0}^{r_{max}}  F(r,\tau) \dr
		& = \int_0^{r_{max}}\int_{\left\{ r \right\} \times S^n} u_r  \dsigma \dvol  +  \int_{0}^{r_{max}} \; F(r,\tau) \Ric_{g(\tau)} \left(\frac{\partial}{\partial r} , \frac{\partial}{\partial r} \right)  \dr, \\  
		&= \int_M u_r  \dvol  +  \int_{0}^{r_{max}} \; F(r,\tau) \Ric_{g(\tau)}\left(\frac{\partial}{\partial r} , \frac{\partial}{\partial r} \right)  \dr, \\ 
		&=  \int_M u_r \dvol + \int_M \; \frac{F(r,\tau)}{\vol(\sphere^n) \psi^n(r,\tau)} \Ric_{g(\tau)}\left(\frac{\partial}{\partial r} , \frac{\partial}{\partial r} \right)  \dvol\\
		& = \int_M \left(u_r + \frac{F(r,\tau)}{\vol(\sphere^n) \psi^n(r,\tau)} \Ric_{g(\tau)}\left(\frac{\partial}{\partial r} , \frac{\partial}{\partial r} \right)\right) \dvol
\end{align}

Ultimately, we want to maximize this quantity. Note that the maximum will be positive because for diffusions whose initial condition is $\delta_{P_+}$ this quantity is positive (remember that our distribution is computed from south pole).\\

First, we try to find a function $F$ with the following properties:
\be
	u_r + \frac{F}{\vol(\sphere^n) \psi^n} \Ric \left(\frac{\partial }{ \partial r} ,  \frac{\partial }{ \partial r} \right) >0
\ee
where, $F(x,\tau)$ given by (\ref{F defn}) is the cumulative distribution of $\nu(x,\tau) = u(x,\tau) \dvol_{(x,\tau)}$. As such, we compute
\begin{align}
	F_r(x,\tau) = \frac{\partial }{\partial r} \int_0^x u(y, \tau) \dvol_{(y,\tau)}  
		&= \frac{1}{\phi(x,\tau)}\frac{\partial }{\partial x} \int_0^x u(y,\tau) \vol(\sphere^n)  \phi(y,\tau) \psi^n(y, \tau) \dy \\ 
		&= \vol(\sphere^n) \psi^n(x,\tau) u(x,\tau) \\ 
		& = \vol(\sphere^n) \psi^n(r,\tau) u(r,\tau).
\end{align}
To simplify notation, set $c = \vol(\sphere^n)$, so then
\be
	u(r,\tau) = \frac{F_r}{c \psi^n}.
\ee
Therefore,
\be
	u_r = \frac{\left(c \psi^n \right)F_{rr} - n c \psi^{n-1} \psi_r F_r}{c^2 \psi^{2n}} = \frac{1}{c \psi^n} \left( F_{rr} - n \frac{\psi_r}{\psi} F_r   \right).
\ee
Hence, again by (\ref{Ricci}),
\begin{align}
	u_r + \frac{F}{\vol(\sphere^n) \psi^n} \Ric \left(\frac{\partial}{ \partial r} ,  \frac{\partial}{ \partial r} \right) &= \frac{1}{c \psi^n} \left( F_{rr}  - n  \frac{ \psi_r}{\psi} F_r - n \frac{\psi_{rr}}{\psi} F  \right) \\ & = \frac{1}{c \psi^n} \left( F_{rr}  - H(r,\tau) F_r  + \Ric \left( \frac{\partial}{\partial r} , \frac{\partial}{\partial r} \right) F  \right);
\end{align}
in which, $H(r, \tau)$ denotes the mean curvature of geodesic spheres of radius $r$ in $(M, g(\tau))$.
This shows that
\begin{align}
\label{eq-to-maxomize}
	\int_M u_r + \frac{F}{\vol(\sphere^n) \psi^n} \Ric \left(\frac{\partial }{ \partial r} ,  \frac{\partial }{ \partial r} \right)  \dvol 
		& = \int_0^{r_{max}} \; \frac{1}{c \psi^n} \left( F_{rr}  - n  \frac{ \psi_r}{\psi} F_r - n \frac{\psi_{rr}}{\psi} F  \right) c \psi^n \dr,  \\ 
		&=  \int_0^{r_{max}} \;  \left( F_{rr}  - n  \frac{ \psi_r}{\psi} F_r - n \frac{\psi_{rr}}{\psi} F  \right) \dr, \\ 
		&= \int_0^{r_{max}} \; \left(  F_{rr} -   H(r, \tau) F_r + F \Ric \left(\frac{\partial}{\partial r} ,  \frac{\partial}{\partial r} \right)\right) \dr.
\end{align}
So finding an estimate boils down to finding a nondecreasing $C^{1,1}$ function $F$ on $[0, r_{max}]$ such that
\be
	F(0,\tau ) = 0\quad \text{ and } \quad F(r_{max}, \tau) = 1,
\ee
and such that $F$ maximizes (\ref{eq-to-maxomize}) (in the ideal case) or, at least, makes (\ref{eq-to-maxomize}) very large. 

From these equations, it seems like we need to choose $F$ in such a way that $H(r, \tau) F_r$ is finite. Since $H(r,\tau) = n \frac{\psi_r}{\psi}$ and $u(r, \tau) = \frac{F_r}{c \psi^n}$,  one way to get an estimate is to take $F_r $ proportional to $\psi^n$. Let $\beta(r)$ be a cut-off function with $\operatorname{supp}(\beta) \subset (r_{max} - 2\delta , r_{max}]$, for some $\delta >0$, and define $F_r$ as the follows:
\[
F_r =
\begin{cases}
            \beta(r) \psi^n(r,\tau) ~~ 	& \text{ for }r_{max} - 2\delta \le r\le r_{max} \\\\
            0  					& \text{ otherwise}.
\end{cases}
\]
This means that we are taking $u = \frac{\beta}{c}$. Then,
\be
	-n\frac{\psi_r}{\psi}F_r = - n \psi^{n-1}\psi_r.
\ee
We will also have
\be
	\int_0^{r_{max}} \; F_{rr}\; \mathrm{d}r = F_r \left( r_{max}, \tau \right) - F_r \left( 0 , \tau\right) = 0.
\ee
Now it only remains to find proper $\delta>0$ and $\beta(r)$. We from the conditions imposed on $\psi$ that  
\be
	\psi(r_{max}) =0 \quad \text{ and } \quad \lim_{r \to r_{max}^-} \psi_r = -1;
\ee
therefore, for any small enough $\delta>0$, we can assume that, 
\be
	0 \le \psi \le 2 \epsilon, \quad \text{ on } [r_{max} - 2\delta , r_{max}],
\ee
and 
\be
	-1 \le \psi_r \le -1 + \lambda, \quad \text{ on } [r_{max} - 2\delta , r_{max}].
\ee
Note that $\lambda$ is small and depending on $\delta$; which means that $\psi$ is strictly monotone on $[r_{max} - 2\delta , r_{max}]$. Now let $\gamma(\rho)$ be a nonegative cut-off function such that $\operatorname{supp}(\gamma) \subset [0 , 2\epsilon)$ and define
\be
	\beta(r) = \frac{- \gamma\left(\psi \right) \psi_r}{\int_{r_{max} - 2\delta}^{r_{max}} \; \gamma\left(\psi \right) \psi^n \psi_r \dr} = \frac{ - \gamma\left(\psi \right) \psi_r}{\int_0^{2 \epsilon} \; \gamma\left( \rho \right) \rho^n  \; \mathrm{d} \rho}.
\ee
Hence,
\be
	F(r) = \int_0^r \; F_r \dr = \frac{\int_0^{r_{max} - 2\delta} \; - \gamma\left(\psi \right) \psi^n \psi_r \dr}{\int_0^{2 \epsilon} \; \gamma \left( \rho \right) \rho^n  \; \mathrm{d} \rho}, ~~\text{ for } r \in  [r_{max} - 2\delta , r_{max}],
\ee
and $F(r) = 0$, for $r \in [0 , r_{max} - 2\delta]$. So,
\be
	\left| \int_0^{r_{max}} \; F \Ric \left(\frac{\partial}{\partial r} , \frac{\partial}{\partial r} \right)\; \mathrm{d}r \right| \le 2\delta \; || \Ric  ||^2,
\ee
and
\be\label{eq-estimate-for-delta}
	\left(1-\lambda \right)\frac{\int_0^{2 \epsilon} \;  \gamma\left( \rho \right) \rho^{n-1} \; \mathrm{d} \rho}{\int_0^{2 \epsilon} \; \gamma \left( \rho \right) \rho^n  \; \mathrm{d} \rho} 			\le  \frac{\int_0^{r_{max}} \; n  \gamma\left( \psi \right) \psi^{n-1} \psi_r^2 \dr }{\int_0^{2 \epsilon} \; \gamma \left( \rho \right) \rho^n  \; \mathrm{d} \rho} 
			= \int_0^{r_{max}} - n \frac{\psi_r}{\psi} F_r \dr.
\ee
Let $\gamma(\rho)$ be a nonincreasing function such that
\[
\gamma(\rho) = 
\begin{cases}
1, \quad 	&\text{ for } \rho \in [0 , \epsilon],\\
0, 		&\text{ for } \rho \geq 2 \epsilon.
\end{cases}
\]

Then, we have
\be
	\left(1-\lambda \right)\frac{\int_0^{2 \epsilon} \;  \gamma\left( \rho \right) \rho^{n-1} \; \mathrm{d} \rho}{\int_0^{2 \epsilon} \; \gamma \left( \rho \right) \rho^n  \; \mathrm{d} \rho} 			\ge \left(1-\lambda \right)\frac{\int_0^\epsilon \rho^{n-1} \; \mathrm{d}\rho}{ \int_0^{2\epsilon} \; \rho^n \; \mathrm{d}\rho} 
			= \left(1-\lambda \right) \frac{\frac{\epsilon^n}{n}}{\frac{(2\epsilon)^{n+1}}{n+1}} 
			= \frac{n+1}{n 2^{n+1} \epsilon}
\ee
which goes to infinity as $\delta \to 0$.\\

This shows that for two diffusions $\nu_1(x , \tau)$ and $\nu_2(x , \tau)$ on $M_1$ and $M_2$ (resp.) such that for $i =1$ or $i=2$ (or both) we have
\be
	\lim_{\tau \to (T'-T)^+} \nu(x,\tau) = \delta_{p_i}; \text{ for } p_i \in M_i,
\ee
then, 
\be
	\frac{\partial}{\partial \tau} \int_{0}^{\diam(M_i)}  F_i(r,\tau) \dr ~~ \text{ is unbounded from below}.
\ee
\hfill \qed\\

Using this heuristic, we can now prove our main theorem. 

\subsection{Proof of Theorem \ref{thm-main-1}}
\label{proof}
In this section we prove our main theorem. Recall, $\AK$ denotes the class of smooth rotationally symmetric metrics on $\sphere^{n+1}$ which have the following properties: (1) positive scalar curvature everywhere, (2) positive sectional curvature on planes tangent to $\{x\} \times \sphere^{n}$, (3) they are ``sufficiently pinched'' meaning the minimum radius should be small relative to the maximum radius, and (4) positive Ricci curvature on the polar caps (the part from the pole to the nearest `bump').

\begin{thm}
Let $g_0 \in \AK$ be an $SO(n+1)$-invariant metric on $\sphere^{n+1}$ and $g(t)$ a solution to the Ricci flow for $t \in [0,T)$, up to some finite time $T<\infty$. Assuming the diameter remains bounded as $t \nearrow T$,  then the neckpinch singularity which develops at $t= T$ occurs only on the totally geodesic hypersurface  of $\{x_0\} \times \sphere^{n}$, for some $x_0 \in (-1,1)$. 
\end{thm}

\begin{proof}
Following the framework described in Section \ref{framework}, consider the metric space $X = (\MM) \cup_h [0,1]$ equipped with the one-parameter family of distance metrics $D(t)$ given by (\ref{D metric}) as well as the family of Hausdorff measures on $X$ denoted $\mu(t)$. Note that for non-singular times $\mu(t)$ coincides with the usual Riemannian volume measure. \\

Take, as before, $\tau := T' - t$ and consider two weak diffusions $\nu_1(\tau)$ and $\nu_2(\tau)$ on $(X, D(t), \mu(t))$ as defined in the metric-measure sense by Definition \ref{weakdiffusion}.\\ 

From (\ref{W1 vs L1 dist}) we know that:
\be
	W_1 \left(\nu_1(\tau) , \nu_2(\tau), \tau \right) =  || F(r,\tau) - G(r,\tau) ||_{L^1(\R)}
\ee
where, $F(r,\tau)$ and $G(r,\tau)$ are the cumulative distribution functions of $ \left( \operatorname{pr}_r \right)_* \nu_1$ (computed from south pole)  and $ \left( \operatorname{pr}_r \right)_* \nu_2$ (computed from north pole) respectively. Since $\nu_1$ and $\nu_2$ have disjoint support, if we consider both distributions to be computed from south pole, then we have:
\begin{align}
	W_1 \left(\nu_1(\tau) , \nu_2(\tau), \tau \right) &= \int_\R   \left|  F(r,\tau)  -   G(r,\tau)  \right| \dr \\ &=  L(\tau) + \int_{0}^{\diam(M_1)}    F(r,\tau) \dr  +  \int_{0}^{\diam(M_2)}   G(r,\tau)  \dr.  
\end{align}
Therefore, to have a weak super solution of the Ricci flow as characterized in Theorem \ref{thm-McCannTopping}, we must have
\be
\label{W1estimate}
	\frac{\partial}{\partial \tau} W_1 \left(\nu_1(\tau) , \nu_2(\tau),\tau \right) = \frac{\partial L}{\partial \tau} + \frac{\partial}{\partial \tau} \int_{0}^{\diam(M_1)}   F(r,\tau) \dr + \frac{\partial}{\partial \tau}  \int_{0}^{\diam(M_2)}  G(r,\tau)  \dr \le 0
\ee
which implies
\be
	\frac{\partial L}{\partial \tau} \le - \frac{\partial}{\partial \tau} \int_{0}^{\diam(M_1)}   F(r,\tau) \dr - \frac{\partial}{\partial \tau}  \int_{0}^{\diam(M_2)}    G(r,\tau)  \dr . 
\ee

However, by Propostion \ref{diambound}, we can find two separate diffusions $\nu_1(\tau)$ and $\nu_2(\tau)$ on $\sphere^{n+1}$ which begin as Dirac distributions on $M_1$ and $M_2$ (resp.); i.e. 
\be
	\lim_{\tau \to T'-T} \nu_i(x,\tau) = \delta_{p_i}, \quad \text{ for } p_i \in M_i.
\ee
Therefore, $\supp(\nu_i(T'-T)) \subset M_i$ and thus, by Lemma \ref{supp lemma}, $\supp(\nu_i(\tau)) \subset M_i$, for all $\tau < T'-T$. As computed in Proposition \ref{diambound}, it follows that for any $M>0$,
\[
\min\left(\frac{\partial}{\partial \tau} \int_{0}^{\diam(M_1)}  F(r,\tau) \dr,   ~~\frac{\partial}{\partial \tau}  \int_{0}^{\diam(M_2)}  G(r,\tau) \dr  \right) >M.
\]
Thus, the only way for (\ref{W1estimate}) to hold is if $\frac{\partial L}{\partial \tau} < -2M$ for any real number $M$. Or, equivalently, that $\frac{\partial L}{\partial t}$ is unbounded from below as $t \nearrow T$. However, this would imply that the diameter is also unbounded in the flow and we arrive at our contradiction. Thus, the neckpinch singularity that arises at time $t = T$ must occur only at a single point; i.e. a hypersurface $\{x_0\} \times \sphere^{n}$, for some $x_0 \in (-1,1)$.\\
 \end{proof}

\subsection{The Asymmetric Case}

There does not yet exist a rigorous forward evolution out of a general (not necessarily rotationally symmetric) neckpinch singularity. However, the formal matched asymptotics found in \cite[Section 3]{AK2} predict that such evolutions should exist.\\

As we have seen in the previous section, the main idea in proving the one-point pinching phenomenon using the McCann-Topping's theory is to use the infinite propagation speed of heat type equations and the fact that heat does not travel along intervals. These facts do not require any symmetry at all. This suggests that our techniques can be used to estimate the size of the neckpinch singularity for a general neckpinch singularity by simply taking diffusions that are approaching $\delta$-Dirac measures at the singular time.

\bibliographystyle{plain}
\bibliography{reference2013}

\end{document}